\documentclass[journal]{IEEEtran}
\ifCLASSINFOpdf
\else
\fi
\usepackage{graphicx}

\usepackage{subdepth} 

\usepackage{graphics} 
\usepackage{epsfig} 

\usepackage{amsthm,amssymb} 
\usepackage{amsmath,psfrag,color,graphicx} 

\usepackage{thmtools}
\renewcommand\thmcontinues[1]{Continued}

\usepackage{graphicx}
\usepackage{quoting} 
\usepackage{lettrine} 
\usepackage{tocloft} 
\usepackage{titletoc}
\usepackage{fmtcount}
\usepackage{calc} 
\usepackage{cite} 
\usepackage{multicol} 
\usepackage{centernot}
\usepackage{nicefrac}
\usepackage{dsfont}
\usepackage{booktabs}
\usepackage{mathtools, cuted}

\usepackage{lipsum}
\usepackage{psfrag}
\usepackage[export]{adjustbox} 

\usepackage{mathtools}

\PassOptionsToPackage{cmyk}{xcolor}
\usepackage[cmyk]{xcolor}
\selectcolormodel{cmyk}

\usepackage{textcomp}

\usepackage{txfonts}
\usepackage{lmodern}
\usepackage{times}

\DeclareFontEncoding{LS1}{}{}
\DeclareFontSubstitution{LS1}{stix}{m}{n}

\makeatletter
\newcommand*\textmathversion{\csname textmv@\math@version\endcsname}
\newcommand*\textmv@normal{m}
\newcommand*\textmv@bold{b}
\makeatother

\newcommand{\ve}[2][]{\ensuremath{\boldsymbol{\mathrm{#2}}}_{#1}}
\newcommand{\vet}[2][]{\ensuremath{\smash{\boldsymbol{\mathrm{#2}}^{\!\top}_{#1}}}}
\newcommand{\vd}[2][]{\ensuremath{\dot{\boldsymbol{\mathrm{#2}}}_{#1}}}

\newcommand{\ma}[2][]{\ensuremath{\boldsymbol{\mathrm{#2}}}_{#1}}
\newcommand{\mat}[2][]{\ensuremath{\boldsymbol{\mathrm{#2}}^{\!\top}_{#1}}}
\newcommand{\md}[2][]{\ensuremath{\dot{\boldsymbol{\mathrm{#2}}}}_{#1}}

\DeclareMathOperator{\skews}{\ensuremath{\mathrm{skew}}}

\DeclareMathOperator{\grad}{\ensuremath{\mathrm{grad}}}

\newcommand{\R}{\ensuremath{\mathds{R}}}

\newcommand{\C}{\ensuremath{\mathds{C}}}

\newcommand{\N}{\ensuremath{\mathds{N}}}

\newcommand{\SOT}{\ensuremath{\mathsf{SO}(3)}}

\newcommand{\SO}{\ensuremath{\mathsf{SO}(n)}}

\newcommand{\St}{\ensuremath{\mathsf{St}(p,n)}}

\newcommand{\ie}{\textit{i.e.}, }

\newcommand{\eg}{\textit{e.g.}, }

\newcommand{\mtr}{\hspace{-0.3mm}\ensuremath{^\top}}

\newcommand{\diff}{\ensuremath{\mathrm{d}}}

\newcommand{\inv}{\ensuremath{^{-1}}}

\newcommand\raiseT[2]{\setbox0\hbox{$#1{#2}$}\raise\dp0\box0}

\newcommand{\G}{\ensuremath{\mathcal{G}}}

\newcommand{\V}{\ensuremath{\mathcal{V}}}

\newcommand{\E}{\ensuremath{\mathcal{E}}}

\newcommand{\M}{\ensuremath{\mathcal{M}}}

\DeclareMathOperator{\im}{\ensuremath{\mathrm{Im}}}

\newcommand{\Ni}{\ensuremath{\mathcal{N}_i}}

\makeatletter
\newcommand{\raisemath}[1]{\mathpalette{\raisem@th{#1}}}
\newcommand{\raisem@th}[3]{\raisebox{#1}{$#2#3$}}
\makeatother

\newcommand{\ts}[2][]{\ensuremath{\mathsf{T}_{#2}#1}}

\newcommand{\AGAS}{\textsc{agas}}

\definecolor{kthbluergb}{RGB}{25,84,166}
\definecolor{kthbluecmyk}{cmyk}{1,0.55,0,0}

\definecolor{kthblueA}{RGB}{25,84,166}
\definecolor{kthblueB}{RGB}{46,124,192}
\definecolor{kthblueC}{RGB}{112,153,209}
\definecolor{kthblueD}{RGB}{164,186,225}
\definecolor{kthblueE}{RGB}{211,220,241}

\newcommand\smallO{
	\mathchoice
	{{\scriptstyle\mathcal{O}}}
	{{\scriptstyle\mathcal{O}}}
	{{\scriptscriptstyle\mathcal{O}}}
	{\scalebox{.7}{$\scriptscriptstyle\mathcal{O}$}}
}

\newcounter{counter} 

\newtheorem{theorem}[counter]{Theorem}
\newtheorem{lemma}[counter]{Lemma}
\newtheorem{proposition}[counter]{Proposition}
\newtheorem{corollary}[counter]{Corollary}

\newtheorem{definition}[counter]{Definition}
\newtheorem{example}[counter]{Example}

\newtheorem{myalgorithm}[counter]{Algorithm}


\newcounter{parentnumber}

\DeclareMathOperator{\conc}{\ensuremath{\ast}}
\DeclareMathOperator{\bigconc}{\ensuremath{\Asterisk}}
\usepackage{mathabx}
\allowdisplaybreaks

\begin{document}
%
\title{Synchronization on Riemannian manifolds: \\Multiply connected implies multistable}


%
%
%

\author{Johan~Markdahl
\thanks{Johan Markdahl (markdahl@kth.se) is with the Luxembourg Centre for Systems Biomedicine at the University of Luxembourg.}
\thanks{Manuscript received Month XY, 2019; revised Month XY, 2019.}}

%
%

\markboth{Journal of \LaTeX\ Class Files,~Vol.~14, No.~8, August~2015}%
{Shell \MakeLowercase{\textit{et al.}}: Bare Demo of IEEEtran.cls for IEEE Journals}
%



\maketitle

\begin{abstract}


This note concerns the evolution of multi-agent systems on networks over Riemannian manifolds. The motion of each agent is governed by the gradient descent flow of a disagreement function that is a sum of (squared) distances between pairs of communicating agents. Two metrics are considered: geodesic distances and chordal distances for manifolds that are embedded in an ambient Euclidean space. We show that networks which, roughly speaking, are dominated by a large cycle yield a multistable systems if the manifold is multiply connected or contains a closed geodesic that is of locally minimum length in a space of closed  curves.  This result summarizes previous results on the stability of splay or twist state equilibria of the Kuramoto model on $\mathcal{S}^1$ and its generalization, the quantum sync model on $\mathsf{SO}(n)$. It also extends them to the Lohe model on $\mathsf{U}(n)$.


\end{abstract}


%
\IEEEpeerreviewmaketitle

%
%
%
%

\section{Introduction}

\label{sec:intro}


\IEEEPARstart{T}{his} note studies how the stability properties of  multi-agent systems on networks over a Riemannian manifold relates to its geometry and topology. 
The focus is on two algorithms which we refer to as refer to as geodesic consensus and chordal consensus \cite{tron2013riemannian,sarlette2009consensus,aydogdu2017opinion}. Both algorithms are gradient descent flows of quadratic disagreement functions, defined in terms of geodesic distances and chordal distances respectively. We show that if the manifold is multiply connected, then both algorithms yield multistable closed-loop systems under  communication topologies that are dominated by one large cycle. The equilibrium configurations are characterized by local coherence between neighboring agents but global incoherence of the system as a whole. This result is interesting since it summarizes previous results in mathematical physics on the multistability of the Kuramoto model and some of its high-dimensional generalizations \cite{wiley2006size,deville2018synchronization}. Moreover, it suggests some rough guidelines for the control problem of which synchronization algorithm to execute on a given manifold. 


To gain an intuitive understanding of these results, imagine a network consisting of a single cycle where the agents are beads that have been  threaded on a string. The string consists of piecewise geodesic curves, each connecting a pair of neighboring agents. Any  configuration where the beads are equidistantly distributed over the string is refered to as a splay state. We can select initial conditions that guarantee the string to be a continuous curve forwards in time. If the manifold is simply connected, then a continuous shortening of the string to a point results in consensus. If the manifold is multiply connected, then, for strings that are not homeomorphic to a point, reaching synchronization requires two neighboring agents to be threaded away from each other. This goes against the design principles of both the chordal and geodesic consensus algorithms. Instead of reaching consensus, the agents will tend to a set of splay states that is asymptotically stable. 

There are several differences between the two algorithms \cite{aydogdu2017opinion}. Chordal consensus requires the Riemannian manifold $\M$ to be embedded in an ambient Euclidean space whereby the Euclidean distance metric can be used. Geodesic consensus can be defined intrinsically without reference to an embedding space. The algorithms are identical when $\M$ is an Euclidean space. For agent configurations $(x)_{i=1}^N\in\M^N$ such that all agents are close to their neighbors, the algorithms can be expected to behave similarly since short chordal distances approximate geodesics. In practice, chordal consensus is prefered over geodesic consensus since chordal distances are easier to calculate than geodesics and squared geodesic distances are non-smooth even on nice manifolds like the circle \cite{conway1996packing}.

Consider the Kuramoto model on networks over the circle in the special case that all frequencies (drift terms) are equal. It is well-known that cycle graphs yield multistability \cite{wiley2006size,dorfler2014synchronization}. They result in a splay or twisted state where the phases of the agents are spread equidistantly over the circle. The chordal consensus algorithm extends the Kuramoto model with equal frequencies to Riemannian manifolds. Generalized twist states are unstable on the high-dimensional Kuramoto model on Stiefel manifolds $\St$ for which $p\leq2n/3-1$ \cite{markdahl2018prx}. However, generalized twist states are stable on $\SO$ \cite{deville2018synchronization}, which is a submanifold of $\mathsf{St}(n,n)$. This apparent disparity is partly explained by our result that multiply connectedness implies multistability. Indeed, the circle and $\SO$ are multiply connected whereas $\St$ is simply connected for $p\leq n-2$. We introduce splay states for geodesic consensus as equidistant partitions of closed geodesics and show that they are asymptotically stable.

The geodesic and chordal consensus algorithms are multistable, wherefore {\it ad-hoc} control designs that yield  almost global asymptotical stable (\AGAS) sync on manifolds were proposed \cite{sarlette2009consensus,tron2012intrinsic}. However, those algorithms are more demanding in terms of computation and sensing. The author showed that the chordal consensus algorithm yields \AGAS{} sync on some specific manifolds \cite{markdahl2018tac,markdahl2018prx}. This paper shows that a manifold being simply connected is a necessary condition for the geodesic and chordal algorithm to yield \AGAS{} sync. However, as we show by  counter-example, simple connectedness is not sufficient. As such the paper provides the following rough guideline for \AGAS{} sync: if the manifold is simply connected then  the geodesic and chordal consensus algorithms can be considered. Otherwise, an {\it ad-hoc} algorithm is preferable.

The main contribution of this note is summarized as follows: (i) we introduce generalized splay states which are equilibria of the geodesic consensus system and show their asymptotical stability under a condition on the geometry of $\M$, (ii) we show that the geodesic and chordal consensus algorithms are multistable for a certain class of graphs under a condition on the topology of $\M$. These are the first results on the geodesic and chordal consensus algorithms (in their full generality \cite{tron2013riemannian,sarlette2009consensus,aydogdu2017opinion}) which concerns equilibria different from consensus. Moreover, they summarize several known results. Multistability of geodesic consensus on $\SOT$ has been shown by simulation \cite{tron2012intrinsic}. Multistability of the chordal consensus has  been established in special cases such as $\mathcal{S}^1$ \cite{wiley2006size} and $\mathsf{SO}(n)$ \cite{deville2018synchronization}. This paper unifies such results and extends them to other systems, most notably to the Lohe model on $\mathsf{U}(n)$ \cite{lohe2010quantum}.

\section{Preliminaries}

\label{sec:problem}

\noindent Let $(\M,g)$ be a complete Riemannian manifold. The set $\M$ is a real, smooth manifold and the metric tensor $g_x$ is an inner product on the tangent space $\ts[\M]{x}$ at $x$. A \emph{closed} manifold is compact and without boundary. All closed manifolds are complete. A manifold $\M$ is \emph{simply connected} if each closed curve can be continuously deformed to a point. A \emph{multiply connected} manifold $\M$ is path connected but contains at least one closed curve which cannot be continuously deformed to a point. See \cite{carmo1992riemannian,jost2008riemannian} for more details on  Riemannian geometry.

The \emph{arc length} of any curve $\gamma:[a,b]\rightarrow\M$ is
\begin{align}
l(\gamma):=\int_{a}^b g_\gamma(\dot{\gamma},\dot{\gamma})^\frac12\diff t.\label{eq:l}
\end{align}
Let $\Gamma(\M)\subset\M$ denote the set of piecewise smooth curves on $\M$. The length of a geodesic curve is the \emph{geodesic distance}
\begin{align*}
d_g(x,y):=\inf\{l(\gamma)\,|\,\gamma\in\Gamma(\M),\,\gamma(a)=x,\gamma(b)=y\}.
\end{align*}
A curve $\gamma(a)=x$, $\gamma(b)=y$ of minimal length is a geodesic (up to parametrization) from $x$ to $y$. Completeness ensures that at least one geodesic exists between any pair of points. The \emph{concatenation} $\gamma=\gamma_1\conc\gamma_2$ of $\gamma_1:[a,b]\rightarrow\M$ and $\gamma_2:[b,c]\rightarrow\M$ with $\gamma_1(b)=\gamma_2(b)$ is the curve  $\gamma|_{[a,b]}=\gamma_1$, $\gamma|_{[b,c]}=\gamma_2$. In this paper, whenever we concatenate two or more curves, it is assumed that the parametrizations match up. 




Let $\Lambda(\M)$ denote the set of \emph{closed curves} on $\M$, \ie curves $\gamma\in\Gamma(\M)$ such that $\gamma(b)=\gamma(a)$. Let $\im\gamma\subset\M$ denote the image of a curve. Introduce an equivalence relation ${\sim}$ on $\Lambda(\M)$, where $\gamma\sim\lambda$ if $\im\lambda=\im\gamma$ and $l(\lambda)=l(\gamma)$. Denote
\begin{align}
[\gamma]&:=\{\lambda\in\Lambda(\M)\,|\,\im\lambda=\im\gamma,\,l(\lambda)=l(\gamma)\},\nonumber\\
\Lambda(\M)/{\sim}&:=\{[\gamma]\subset\Lambda(\M)\,|\,\gamma\in\Lambda(\M)\}.\label{eq:quotient}
\end{align}
The \emph{Hausdorff distance} between two sets $\mathcal{Y},\mathcal{Z}\subset\mathcal{X}$, where $(\mathcal{X},d)$ is a metric space, is
\begin{align*}
d_H(\mathcal{Y},\mathcal{Z}):=\max\big\{\!\adjustlimits\sup_{y\in\mathcal{Y}}\inf_{z\in\mathcal{Z}}d(y,z),\,\adjustlimits\sup_{z\in\mathcal{Z}}\inf_{y\in\mathcal{Y}}d(y,z)\big\}.
\end{align*}
We let $(\M,d_g)$ be the metric space. Define another metric
\begin{align*}
d_{\sim}(\gamma,\lambda):=d_H([\gamma],[\lambda])+|l(\gamma)-l(\lambda)|
\end{align*}
such that $d_{\sim}(\gamma,\lambda)=0$ for all $\lambda\in[\gamma]$. Then $(\Lambda(\M)/{\sim},d_{\sim})$ is a metric space which admits the notions of optimization.

\begin{proposition}[Klingenberg \cite{klingenberg1978lectures}]\label{th:simply}
Assume that the Riemannian manifold $(\M,g)$ is closed and multiply connected. Then $(\mathcal{M},g)$ contains a closed geodesic that is a local minimizer of the curve length function $l$ over $\Lambda(\M)/{\sim}$.
\end{proposition}

\begin{example}
The manifold $\smash{\R^2}\backslash\{\ve{0}\}$ is multiply connected, yet it does not contain a closed geodesic of locally minimum length. Proposition \ref{th:simply} does not apply since $\smash{\R^2}\backslash\{\ve{0}\}$ is open. 
\end{example}


\begin{example}\label{exa:cont}
The torus is multiply connected. A closed curve around the torus tube cannot be continuously deformed to a point. If that curve is a circle, then it is a local minimizer of $l$ in the space of closed curves. The sphere $\mathcal{S}^2$ is simply connected. The closed geodesics on $\mathcal{S}^2$ are great circles, \eg the equator. The equator is not a local minimizer of $l$ since there are closed curves of constant latitude arbitrarily close to the equator that are shorter than it, see Fig. \ref{fig:manifolds}. On the capsule, in the regions where the cylinder and hemispheres meet, there are curves which are saddle points of $l$. They are minimizers of $l$ on the cylinder but not on the hemispheres. On both the torus and the peanut there is a single closed curve which is a strict local minimizer of $l$. The torus is multiply connected but the peanut is simply connected. Simple connectedness does not rule out the existence of a closed curve of minimum length.
\end{example}

\begin{figure}[htb!]
\centering	\includegraphics[width=0.45\textwidth]{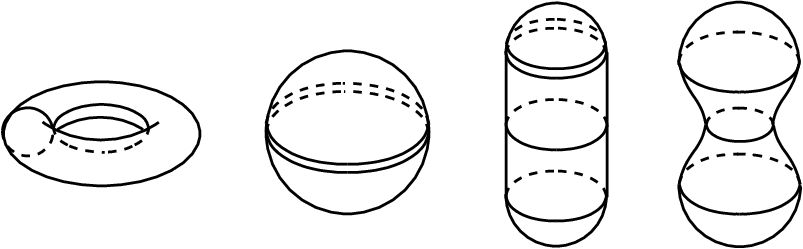}
\caption{A torus, a sphere, a capsule, and a peanut.}
\label{fig:manifolds}
\end{figure}

Assume that the manifold is \emph{geodesically complete} which implies that there exists at least one geodesic path between any two points $x,y\in\M$. Moreover, assume that for some \emph{open neighborhood} of $x$,  $\mathcal{B}_\varepsilon(x):=\{z\in\M\,|\,d_g(x,z)<\varepsilon\}$, there exists a unique geodesic from $x$ to each $y\in\mathcal{B}_\varepsilon(x)$. The largest value $\varepsilon>0$ for which this holds is the \emph{injectivity radius} $r(x)$ at $x$. We assume that $R:=\inf_{x\in\M}r(x)>0$. 

Let $\exp_x:\ts[\M]{x}\rightarrow\M$ denote the \emph{exponential map}. Given a point $x\in\M$ and a tangent vector $v\in\ts[\M]{x}$,  $\exp_x(v)$ yields the point $y\in\M$ that lies at a distance $g_x(v,v)^\frac12$ from $x$ along the geodesic through $x$ with $v$ as a tangent vector. Let $\mathcal{S}_x\subset\ts[\M]{x}$ be the largest open, path connected set containing $0$ on which $\exp_x$ is a diffeomorphism. Denote $\mathcal{X}_x:=\exp_x(\mathcal{S}_x)\subset\M$. Note that the injectivity radius $r:\M\rightarrow\R$ is the radius of the largest geodesic ball $\mathcal{B}_{r}(x)$ contained in $\mathcal{X}_x$. The inverse of the exponential map is well-defined on $\mathcal{X}_x$. It is the logarithm map  $\smash{\log_x}:\mathcal{X}_x\rightarrow\ts[\M]{x}$ given by $\log_x:\exp_x(v)\mapsto v$.

The directional derivative of a smooth function $f:\M\rightarrow\R$ at $x\in\M$ along $v\in\ts[\M]{x}$ is given by $\tfrac{\diff}{\diff t}f(\gamma(t))|_{t=0}$, where $\gamma\in\Gamma$ satisfies $\gamma(0)=x$, $\dot{\gamma}(0)=v$. The \emph{intrinsic gradient} of $f$ is defined as the vector $\grad f(x)\in\ts[\M]{x}$ which satisfies
\begin{align*}
g_x(\grad f(x),v)=\tfrac{\diff}{\diff t}f(\gamma(t))|_{t=0}
\end{align*}
for all $v\in\ts[\M]{x}$. In particular, it holds that $\grad d^2_g(x,y)=-2\log_x(y)$ for all $y\in\mathcal{X}_x$.


\section{Gradient flows on Riemannian manifolds}

\noindent 
Let $V:\M\rightarrow\R$ be a $C^2$ function on a Riemannian manifold. The \emph{gradient descent flow} of $V$ on $\M$ is given by
\begin{align}\label{eq:first}
\dot{x}=-\grad V(x),
\end{align}
for any $x(0)\in\M$. The solutions $x(t)$, $t\in\R$, to \eqref{eq:first} are refered to as \emph{flow lines}. Note that the \emph{equilibria} of \eqref{eq:first} are the \emph{critical points} of $V$, \ie the points for which $\grad V(x)=0$. It may hence be advantageous to adopt an optimization perspective. The relation between the stability properties of the equilibria of a gradient descent flow and the critical points of the potential function $V$ is somewhat complicated, so we need to define precise notions to specify it. See \cite{absil2006stable} for more details about these issues for gradient descent flows on $\R^n$.


\begin{definition}\label{def:minimizer}
A set $\mathcal{S}\subset\mathcal{X}$ of minimizers of a real function $f:\mathcal{X}\rightarrow\R$ from a metric space $(\mathcal{X},d_H)$ is said to be a \emph{local minimizer} if for some $\varepsilon>0$ there is an open neighborhood  $\mathcal{B}_\varepsilon(\mathcal{S})=\{x\in\M\,|\,d_H(x,\mathcal{S})<\varepsilon\}$ such that $f|_{\mathcal{S}}\leq f(x)$ for all $x\in\mathcal{B}_\varepsilon(\mathcal{S})$. Moreover, if the inequality is strict for all $x\in\mathcal{B}_\varepsilon(\mathcal{S})\backslash\mathcal{S}$, then $\mathcal{S}$ is said to be a \emph{strict local minimizer}.
\end{definition}

\begin{definition}\label{def:critical}
A set $\mathcal{S}\subset\mathcal{X}$ of minimizers of a real function $f:\mathcal{X}\rightarrow\R$ from a metric space $(\mathcal{X},d_H)$ is said to be \emph{isolated critical} if for some $\varepsilon>0$ there is an open neighborhood  $\mathcal{B}_\varepsilon(\mathcal{S})=\{x\in\M\,|\,d_H(x,\mathcal{S})<\varepsilon\}$ such that $\mathcal{B}_{\varepsilon}(\mathcal{S})\backslash\mathcal{S}$ is void of critical points.
\end{definition}

\begin{example}[continues=exa:cont]
Consider the manifolds in Fig. \ref{fig:manifolds}. The torus and peanut have closed geodesics that are strict minimizers of $l$ on $\Lambda(\M)/{\sim}$. The sphere and capsule do not.
\end{example}



\begin{proposition}\label{prop:lines}
Let $\M$ be closed and take any $V:\M\rightarrow\R$ that is $\smash{C^2}$. The flow $\dot{x}=-\grad V$ has a unique solution $x(t)\in\M$ which exists for all $t\in\R$. The potential function $V$ decreases along flow lines of $\dot{x}=-\grad V$. For any flow line, $\grad V(x(t))\rightarrow 0$ as $t\rightarrow\pm\infty$.
\end{proposition}

\begin{proof}
See \cite{jost2008riemannian}.
\end{proof}




\begin{proposition}[Lyapunov theorem]\label{prop:not}
Let $\M$ be closed and take any $V:\M\rightarrow\R$ that is $\smash{C^2}$. Let $\mathcal{S}$ be a compact set of local minimizers of $V$. If $\mathcal{S}$ is a strict local minimizer, then $\mathcal{S}$ is a Lyapunov stable equilibrium set of $\dot{x}=-\grad V$. If $\mathcal{S}$ is also isolated critical, then it is asymptotically stable.
\end{proposition}

\begin{proof}
By definition of $\mathcal{S}$ being a strict local minimizer, for any sufficiently small $\varepsilon\in[0,\infty)$, there exists a closed geodesic ball $\mathcal{B}_\varepsilon(\mathcal{S})\subset\M^N$ of radius $\varepsilon$ such that $V|_{\mathcal{S}}<V(x)$ for all $x\in{\mathcal{B}_\varepsilon(\mathcal{S})\backslash\mathcal{S}}$. Let $\partial\mathcal{B}_\varepsilon(\mathcal{S})$ denote the boundary of $\mathcal{B}_\varepsilon(\mathcal{S})$. Let $\alpha:=\inf_{y\in\partial\mathcal{B}_\varepsilon(\mathcal{S})}V(y)>V|_{\mathcal{S}}\geq0$. Let
\begin{align*}
\mathcal{A}:=\{y\in\mathcal{B}_\varepsilon(\mathcal{S})\,|\,V(y)<\alpha\}.
\end{align*}
There exists a $\delta>0$ such that $\mathcal{B}_\delta(\mathcal{S})\subset\mathcal{A}$. The potential function $V$ decreases along flow lines by Proposition \ref{prop:lines}. If a flow line starting in $\mathcal{B}_\delta(\mathcal{S})$ passed through $\partial\mathcal{B}_\varepsilon(\mathcal{S})$, then $V$ would have had to have increased, a contradiction.

Assume that $\mathcal{S}$ is also an isolated critical set. The flow lines are contained in a compact set $\mathcal{B}_\varepsilon(\mathcal{S})$. Each line approaches a set of critical points by Proposition \ref{prop:lines}. All of the critical points in $\mathcal{B}_\varepsilon(\mathcal{S})$ belong to $\mathcal{S}$ by assumption.\end{proof}


A subset $\mathcal{S}\subset\M$ has \emph{measure zero} if for every smooth chart $(\mathcal{U},\varphi)$ in an atlas of $\M$, the set $\varphi(\mathcal{U}\cap\mathcal{S})$ has Lebesgue measure zero in $\R^n$. A subset $\mathcal{U}\subset\M$ has \emph{strictly positive} Riemannian measure if it does not have measure zero. Let $\Phi(t,x_0)$ denote the \emph{flow} of \eqref{eq:first}, \ie the solution $x(t)$ of  \eqref{eq:first} at time $t$ for the initial condition $x(0)=x_0$. Define $\Omega(\mathcal{S})$
\begin{align*}
\Omega(\mathcal{S}):=\{\omega\in\M\,|\,\omega=\lim_{t\rightarrow\infty}\Phi(t,x_0),\,x_0\in\mathcal{S}\}.
\end{align*}

\begin{definition}[Multistable]\label{def:multistable}
Let $\M$ be a closed manifold. The gradient flow \eqref{eq:first} is \emph{multistable} if there exists two sets $\mathcal{S}_1,\mathcal{S}_2\subset\M$ of strictly positive Riemannian measure such that no points of $\Omega(\mathcal{S}_1)$ can be path connected to $\Omega(\mathcal{S}_2)$ via a path in $\Omega(\M)$.\end{definition}



\section{Synchronization on Riemannian manifolds}

\noindent Consider a network of $N$ interacting agents. The interaction topology is modeled by an undirected graph $\mathcal{G}:=(\mathcal{V},\mathcal{E})$ where the nodes $\V:=\{1,\ldots,N\}$ represent agents and an edge $e=\{i,j\}\in\E$ indicates that agent $i$ and $j$ can communicate. Assume that the graph is connected, whereby there is at least an indirect path of communication between any two agents. In this note, we mainly focus on the cycle graph
\begin{align}\label{eq:HN}
\mathcal{H}_N:=(\V,\E):=(\{1,\ldots,N\},\{\{i,i+1\}\,|\,i\in\V\}).
\end{align}
For notational convenience we use modular arithmetic $N+1\equiv1\,(\mathrm{mod}\,N)$ when adding the indices of $\mathcal{H}_N$, \ie $\{1,N\}\in\E$.

The state $x_i$ of agent $i$ belongs to the manifold $\M$. The states are grouped together in a tuple, $x:=(x_i)_{i=1}^N\in\M^N$.  The \emph{consensus manifold} $\mathcal{C}$ of a  Riemannian manifold $(\mathcal{M},g)$ is the diagonal space of $\M^N$ given by the set
\begin{align}\label{eq:C}
\mathcal{C}:=\{(x_i)_{i=1}^N\in\mathcal{M}^N\,|\,x_i=x_j,\,\forall\,\{i,j\}\in\E\}, 
\end{align}
where it is assumed that $\G$ is connected. The consensus set is a Riemannian manifold; in fact, it is diffeomorphic to $\mathcal{M}$ by the map $\smash{(x_i)_{i=1}^N}\mapsto x_1$. The terms synchronization and consensus are interchangeable in this note: 



\begin{definition}
The gradient flow \eqref{eq:first} is said to \emph{synchronize}, or equivalently, to reach \emph{consensus}, if $\lim_{t\rightarrow\infty}d_H(x(t),\mathcal{C})=0$.
\end{definition}

This note mainly concerns the local behaviour of a multi-agent system where the distance $d_g(x_i,x_j)$ between any pair of interacting agents can be made arbitrarily small by increasing $N$. As such, we work on subsets of the manifold where all geodesics are uniquely defined. Under these circumstances, define the notion of a \emph{closed broken geodesic}:
\begin{definition}[Closed broken geodesic]\label{def:broken} Let $x=(x_i)_{i=1}^N\in\M^N$ denote agent positions and $\gamma_i:[t_i,t_{i+1}]\rightarrow\M$ be a geodesic curve with $\gamma_i(t_i)=x_i$, $\gamma_i(t_{i+1})=x_{i+1}$ (using the convention $N+1=1\,(\emph{mod}\, N)$). By a \emph{closed broken geodesic} interpolating $x$ we refer to the closed curve $\gamma_x:[t_1,t_N]\mapsto\M$ given by the concatenation $\gamma_x=\bigconc_{i=1}^{N}\gamma_i$.
\end{definition}	
%
\noindent A closed broken geodesic can be intuitively grasped by imagining the agents as being beads on string, see Section \ref{sec:intro}. Closed broken geodesics will be used to represent cycle networks $\mathcal{H}_N$.

\subsection{Two gradient descent flows}

\noindent This paper concerns two synchronization algorithms, which we refer to as geodesic consensus  and chordal consensus. Both algorithms are gradient descent flows of \emph{disagreement functions}, \ie potential functions $W:\M^N\rightarrow\R$ of the form
\begin{align}\label{eq:W}
W(x):=\tfrac12\sum_{\{i,j\}\in\E} w_{ij}\int_0^{d^2(x_i,x_j)}f(s)\,\diff s
\end{align}
where $w_{ij}\in(0,\infty)$ are weights, $w_{ij}=w_{ji}$, $d$ is a metric, $f:\R\rightarrow[0,1]$ is a \emph{smoothing function} to be specifed, and $\G=(\V,\E)$ is an undirected graph which represents the network. The geodesic algorithm uses the geodesic distance $d_g$ as metric whereas the chordal algorithm uses the chordal distance $d_c$, \ie the Euclidean distance in an ambient Euclidean space.

The consensus-seeking multi-agent system on $\mathcal{M}^N$ obtained from the gradient descent flow $W$ is 
\begin{align}
\dot{x}=-\grad W(x), \quad (\dot{x}_i)_{i=1}^N=(-\nabla_i W(x))_{i=1}^N,\label{eq:descent}
\end{align}
where $x_i(0)\in\mathcal{M}$ and $\nabla_i$ denotes the gradient with respect to $x_i$ obtained by holding the other agent states constant. Agent $i$ does not have access to $W$, but can calculate
\begin{align*}
W_i(x_i,x_{j_1(i)},\ldots,x_{j_{|\mathcal{N}_i|}(i)}):=\tfrac12\sum_{j\in\Ni}w_{ij}\int_0^{d^2(x_i,x_j)}f(s)\diff s,
\end{align*}
where $\mathcal{N}_i:=\{j\in\V\,|\,\{i,j\}\in\E\} =:\{x_{j_1(i)},\ldots,x_{j_{|\mathcal{N}_i|}(i)}\}$ is the set of \emph{neighbors} of $i$. Symmetry of $d_g$ gives $W=\tfrac12\sum_{i\in\V}W_i$ whereby it follows that $\nabla_i W_i=\nabla_i W$. 

From a control design perspective, we assume that the dynamics of each agent take the form $\dot{x}_i=u_i$ with $u_i\in\ts[\mathcal{M}]{x_i}$. Furthermore, we assume that agent $i$ is equipped with sensors that allow it to calculate $W_i$ in some small neighborhood around its current position. It follows that agent $i$ can also calculate $u_i:=-\nabla_i W_i$. Note that it requires more information and is more computationally demanding to calculate $d_g$ compared to $d_c$ \cite{aydogdu2017opinion}.

Non-uniqueness of geodesic curves results in a loss of smoothness and the gradient of $V(x):=\smash{W(x)|_{d=d_g}}$ being undefined. Following \cite{aydogdu2017opinion}, in Section \ref{sec:geodesic} we design  $f$ to make $V$ a $C^2$ function everywhere. Note that this issue does not arise in chordal consensus. For chordal consensus we choose $f(x)=1$ whereby the potential simplifies as
\begin{align*}
U(x):=W|_{d=d_c,f=1}=\tfrac12\sum_{\{i,j\}\in\E} w_{ij}{d_c^2(x_i,x_j)}.
\end{align*}
The potential $U$ is smooth for any smooth manifold. The chordal metric is sometimes preferable over the geodesic metric for this reason, see \cite{conway1996packing} for further discussion.



Another distinction can be made between \emph{intrinsic consensus} and \emph{extrinsic consensus}, referring to the concepts of intrinsic and extrinsic geometry.  By intrinsic consensus, we refer to a consensus algorithm defined on a manifold $\M$ that is an abstract topological space. By extrinsic consensus we refer to an algorithm that is defined on a manifold $\mathcal{M}$ that is embedded in an ambient Euclidean space $\R^{n\times m}$. The geodesic consensus algorithm can be either intrinsic or extrinsic. The chordal consensus algorithm can only be used in an extrinsic setting where the chordal distance is defined.

\subsection{Geodesic consensus}\label{sec:geodesic}

\noindent Recall that $R$ denotes the injectivity radius of $\M$. Define
\begin{align}\label{eq:V}
V(x)&:=\tfrac12\!\!\sum_{\{i,j\}\in\E}\!\!w_{ij}\int_0^{d_g^2(x_i,x_j)}f(s)\,\diff s,\\
f(s)&:=\begin{cases} 1, &  s\in [0,(R-\varepsilon)^2),\\
h(s), & s\in [(R-\varepsilon)^2,R^2),\\
0, &  s\in[R^2,\infty),
\end{cases}\label{eq:f}
\end{align}
where $\varepsilon>0$ is a small constant and $h(s)$ is a smooth function that interpolates $f((R-\varepsilon)^2)=1$ and $f(R^2)=0$ in a manner such that $V$ is $C^2$ on $\M^N$. Define:
\begin{myalgorithm}[Geodesic consensus]\label{algo:intrinsic}
The closed-loop system for $\dot{x}_i=u_i$ under the feedback $u_i:=-\nabla_i V_i$ is given by
\begin{align}\label{eq:dxidt}
\dot{x}_i=\sum_{j\in\Ni}w_{ij}f(d_g(x_i,x_j)^2)\log_{x_i}(x_j),\,\forall\,i\in\V.
\end{align}
\end{myalgorithm}

Algorithm \ref{algo:intrinsic} is a \emph{bounded confidence} opinion consensus model  \cite{proskurnikov2018tutorial}, \ie agent $i$ and $j$ cannot influence each other when $d_g(x_i,x_j)\geq R$. Note that Algorithm \ref{algo:intrinsic} simplifies as
\begin{align}\label{eq:dxidt2}
\dot{x}_i=\sum_{j\in\Ni}w_{ij}\log_{x_i}(x_j)
\end{align}
for $d_g(x_i,x_j)\in[0,R-\varepsilon)$.  In Section \ref{sec:main} we show that there exists a forward invariant set $\mathcal{S}$ such that if $x(0)\in\mathcal{S}$, then the dynamics \eqref{eq:dxidt} are of the form \eqref{eq:dxidt2} and remain on that form at all future times. Because of this result, in the rest of this note, our focus with respect to the geodesic consenus algorithm is on the dynamics \eqref{eq:dxidt2} and the potential
\begin{align}\label{eq:V2}
V(x)|_{\mathcal{S}}=\tfrac12\!\!\sum_{\{i,j\}\in\E}\!\!w_{ij}d_g^2(x_i,x_j),
\end{align}
which we refer to as the \emph{simplified forms}.



\subsection{Chordal consensus}

\label{sec:extrinsic_intro}

\noindent Let the manifold $(\M,g)$ be embedded in an ambient Euclidean space  $\R^{n\times m}$, where $g$ is the induced metric. Denote $\ma[i]{X}:=\iota(x_i)$, where $\iota:\M\hookrightarrow\R^{n\times m}$ is the inclusion map. Denote  $\ma{X}:=(\ve[i]{X})_{i=1}^N\in\smash{(\R^{n\times m})^N}$. Given $\ma{X}\in\M$, any matrix $\ma{M}\in\R^{n\times m}$ can be uniquely decomposed as $\ma{M}=\Pi_{\ma{X}}\ma{M}+\Pi^\perp_{\ma{X}}\ma{M}$, where $\Pi_{\ma{X}}$ and $\Pi_{\ma{X}}^\perp$ are orthogonal projections on the tangent space $\ts[\M]{\ma{X}}$ and normal space  $(\ts[\M]{\ma{X}})^\perp$, respectively.

Define $U:\mathcal{U}\rightarrow\R$, where $\mathcal{U}$ is an open neighborhood of $\M^N$ in $\R^{n\times m}$, based on the Euclidean or \emph{chordal distance}
\begin{align}\label{eq:U}
U(\ma{X}):=\tfrac12\sum_{\{i,j\}\in\E}w_{ij}\|\ve[i]{X}-\ve[j]{X}\|^2.
\end{align}
Let $U_i(\ma[i]{X},\ma[j_1(i)]{X},\ldots,\ma[j_{|\mathcal{N}_i|}(i)]{X}):=\tfrac12\sum_{j\in \Ni}w_{ij}\|\ve[i]{X}-\ve[j]{X}\|^2$, where $\{\ma[j_1(i)]{X},\ldots,\ma[j_{|\mathcal{N}_i|}(i)]{X}\}:=\mathcal{N}_i$. The gradient is
\begin{align*}
\nabla_i U=\Pi_{\ma[i]{X}}\tfrac{\partial}{\partial \ma[i]{X}}U_i,
\end{align*}
where $\nabla_i$ denotes the gradient on $\M$ with respect to $\ma[i]{X}$  and $\smash{\tfrac{\partial}{\partial \ma[i]{X}}}$ is the extrinsic gradient in the ambient Euclidean space with respect to $\ma[i]{X}$ \cite{absil2009optimization}. 

The gradient descent flow on $U$ is given by:
\begin{myalgorithm}[Chordal consensus]\label{algo:extrinsic}
The closed-loop system for $\md[i]{X}:=\ma[i]{U}$ under the feedback $\ma[i]{U}:=-\nabla_i U_i$ is given by
\begin{align}\label{eq:dXidt}
\md[i]{X}=\Pi_{\ma[i]{X}}\sum_{j\in\Ni}w_{ij}(\ma[j]{X}-\ma[i]{X}),\,\forall\,i\in\V.
\end{align}
\end{myalgorithm}

Note that if all $\ma{X}\in\M$ have constant norm, $\|\ma {X}\|=k$, then $\Pi_{\ma{X}}\ma{X}=\ma{0}$, whereby the dynamics \eqref{eq:dXidt} simplify to
\begin{align}
\md[i]{X}=\Pi_{\ma[i]{X}}\sum_{j\in\Ni}w_{ij}\ma[j]{X}.\label{eq:kuramoto}
\end{align}
The assumption of $\|\ma[i]{X}\|=k$ implies that the manifold can be embedded in $\R^{n\times m}$ as a subset of a sphere $\mathcal{S}^{nm-1}$ with radius $k$. This is trivially true of the $n$-sphere. Some important such spaces are the Stiefel and Grassmannian manifolds. Examples of systems of the form \eqref{eq:kuramoto} includes: the Kuramoto model on $\mathcal{S}^1$ in Cartesian coordinates (see Example \ref{ex:Kuramoto}), the Lohe model on the $n$-sphere \cite{lohe2010quantum,ha2018relaxation,markdahl2018tac}, the quantum sync model on $\SO$ \cite{deville2018synchronization}, the Kuramoto model on the Stiefel manifold $\mathsf{St}(p,n)$ \cite{markdahl2018tac}, and the Lohe model on $\mathsf{U}(n)$ (see Example \ref{exa:lohe}). 

\begin{example}\label{exa:lohe}
Consider the Lohe model of synchronization on $\mathsf{U}(n)\subset\C^{n\times n}$ \cite{lohe2010quantum},
\begin{align}\label{eq:Lohe}
\md[i]{U}=\ma[i]{H}\ma[i]{U}+\tfrac12\sum_{j\in\mathcal{N}_i}w_{ij}(\ma[j]{U}-\ma[i]{U}\ma[j]{U}^*\ma[i]{U}),
\end{align}
where $\ma[i]{U}\in\mathsf{U}(n)$, $\ma[i]{H}$ is a Hermitian matrix,  $w_{ij}=w_{ji}\in\R$, and $\cdot^*$ denotes complex conjugation. The case of $\ma[i]{H}=\ma{H}$ can be reduced to $\ma[i]{H}=\ma{0}$ after a linear change of coordinates.

Embed $\mathsf{U}(n)$ in $\mathsf{SO}(2n)$ using the \emph{realification map}
\begin{align*}
\varphi:\C^{n\times n}\rightarrow\R^{2n\times2n}:\ma{Z}=\ma{X}+i\ma{Y}\mapsto\begin{bmatrix} \ma{X} &-\ma{Y}\\
\ma{Y} & \ma{X}
\end{bmatrix},
\end{align*}
which is a linear homomorphism, \ie
\begin{align*}
\varphi\,(\alpha\ma[1]{Z}+\beta\ma[2]{Z})&=\alpha\varphi\,(\ma[1]{Z})+\beta\varphi\,(\ma[2]{Z}),\\ 
\varphi\,(\ma[1]{Z}\ma[2]{Z})&=\varphi\,(\ma[1]{Z})\varphi\,(\ma[2]{Z}),
\end{align*}
for any $\alpha,\beta\in\R$ and $\ma[1]{Z},\ma[2]{Z}\in\C^{n\times n}$  \cite{fomenko1995symplectic}. Let $\ma[i]{R}:=\varphi\,\ma[i]{U}$. Note that $\varphi\,\mathsf{U}(n)\subset\mathsf{SO}(2n)$ since $\ma[2n]{I}=\varphi\,\ma{U}^*\ma{U}=(\varphi\,\ma{U})\mtr\varphi\,\ma{U}$ for any $\ma{U}\in\mathsf{U}(n)$, whereby $\det\varphi\,\ma{U}\in\{-1,1\}$ and $\det\varphi\,\ma{U}=\det\varphi\,\ma[n]{I}=1$ by continuity of $\varphi$ and $\det$.

Consider the dynamics of $\ma[i]{R}$ in the case of $\ma[i]{H}=0$,
\begin{align}
\md[i]{R}&=\varphi\,\md[i]{U}=\tfrac12\sum_{j\in\Ni}w_{ij}(\varphi\,\ma[j]{U}-\varphi\,(\ma[i]{U}\ma[j]{U}^*\ma[i]{U}))\nonumber\\
&=\tfrac12\sum_{j\in\Ni}w_{ij}(\ma[j]{R}-\ma[i]{R}\mat[j]{R}\ma[i]{R})=\Gamma_{\ma[i]{R}}\sum_{j\in\Ni}w_{ij}\ma[j]{R}\nonumber\\%
&=\Pi_{\ma[i]{R}}\sum_{j\in\Ni}w_{ij}\ma[j]{R}\label{eq:LoheSO2n}
\end{align}
where $\Gamma_{\ma[i]{R}}:\R^{2n\times 2n}\rightarrow\ts[\mathsf{SO}(2n)]{i}:\ma{M}\mapsto\ma[i]{R}\skews(\mat[i]{R}\ma{M})$ and $\Pi_{\ma[i]{R}}:\R^{2n\times 2n}\rightarrow\ts[\varphi\, \mathsf{U}(n)]{i}$. The last equality can be established based on the fact that $\varphi\,\mathsf{U}(n)=\mathsf{SO}(2n)\cap\varphi\,\mathsf{GL}(n,\C)$ \cite{fomenko1995symplectic}. Note that \eqref{eq:LoheSO2n} is the gradient descent flow \eqref{eq:kuramoto} of $V((\ma[i]{R})_{i=1}^N):=\tfrac12\sum_{j\in\Ni}w_{ij}\|\ma[j]{R}-\ma[i]{R}\|^2$ on $\varphi\,\mathsf{U}(n)$.  
\end{example}



\section{Main results}
\label{sec:main}

\noindent This section consists of two parts, \ref{sec:intrinsic} and \ref{sec:extrinsic}. First, we define splay states of the geodesic consensus system over cycle graphs and show that they are asymptotically stable. Second, we show that the geodesic and chordal systems are multistable over a larger family of graphs if $\M$ is  multiply connected. 


\subsection{Splay states}

\label{sec:intrinsic}

\noindent Splay or twisted states are two terms used to refer to one set of equilibria of the Kuramoto model \cite{wiley2006size,dorfler2014synchronization}. In this note we use splay states when referring to configurations of the geodesic consensus system and twisted states when referring to configurations of the chordal consensus algorithm.
\begin{definition}[Splay state]\label{eq:splay}
Consider the system \eqref{eq:dxidt} on the cycle graph \eqref{eq:HN}. Assume there is a unique geodesic $\gamma_i$ from $x_i$ to $x_{i+1}$.  Assume $\gamma:=\bigconc_{i=1}^N\gamma_i$ is a closed geodesic. 
The system is said to be in a \emph{splay state} if 
\begin{align*}
d_g(x_i,x_{i+1}):=l(\gamma_i)=\frac{w_{i,i+1}^{-1}}{\sum_{j=1}^N w_{j,j+1}^{-1}}l(\gamma)<R-\varepsilon,\,\forall\,i\in\V,
\end{align*}
for some $\varepsilon>0$.
\end{definition}

Note that if $w_{i,i+1}$ does not depend on $i$, then any splay state satisfies $d_g(x_i,x_{i+1})=l(\gamma)/N$. The interpretation is that the positions of the agents partitions the closed geodesic $\gamma$ into $N$ consecutive segments of equal length. 

\begin{example}\label{ex:Kuramoto}
The Lohe model on $\mathcal{S}^n$ with $\mathcal{H}_N$, $w_{ij}=1$ is
\begin{align*}
\vd[i]{x}&=\ma[i]{\Omega}\ve[i]{x}+\Pi_{\ma[i]{x}}\!\!\sum_{j\in\Ni}\!\ve[j]{x}=\ma[i]{\Omega}\ve[i]{x}+(\ma{I}-\ve[i]{x}\vet[i]{x})(\ve[i+1]{x}+\ve[i-1]{x}),
\end{align*}	
where $\ve[i]{x}\in\mathcal{S}^n\subset\R^{n+1}$, $\ma[i]{\Omega}\in\mathsf{so}(n+1)$ \cite{lohe2010quantum}. The Kuramoto model in polar coordinates is obtained for $n=1$ with  $\ve[i]{x}=[\cos\vartheta_i\,\sin\vartheta_i]\mtr$, $\omega_i=\ma[i,12]{\Omega}$, as
\begin{align}
\dot{\vartheta}_i&=\omega_i+\sin(\vartheta_{i+1}-\vartheta_i)+\sin(\vartheta_{i-1}-\vartheta_i).\label{eq:theta}
\end{align}	
If $\ma[i]{\Omega}=\ma{\Omega}$ (or $\omega_i=\omega$), then the change of variables $\ve[i]{y}=\exp(-\ma{\Omega}t)\ve[i]{x}$ (or  $\theta_i=\vartheta_i-\omega t$) transforms the system to  \eqref{eq:kuramoto}.
The sets of $q$-twisted states where $k\in\N$, are given by
\begin{align*}
\mathcal{T}_{\ve{y}}&=\{(\ve[i]{y})_{i=1}^N\in(\R^{n+1})^N\,|\,\ve[i+1]{y}=\ma{R}\ve[i]{y}\},\\
\mathcal{T}_{\theta}&=
\{(\theta)_{i=1}^N\in\R^N\,|\,\theta_i=2iq\pi/N+\varphi \,,\,\varphi\in\R\},
\end{align*}
where $\ma{R}\in\mathsf{SO}(n+1)$ is a rotation by $2q\pi/N$ radians.
\end{example}

Twisted states are equilibria of the chordal consensus system on $\mathcal{S}^n$. Twisted states are unstable for $n\geq2$ \cite{markdahl2018tac}, but asymptotically stable on $\mathcal{S}^1$ \cite{wiley2006size}. Those results are consistent with this note since $\mathcal{S}^n$ is simply connected for $n\geq2$, but $\mathcal{S}^1$ is multiply connected. Twisted states and other related configurations are of relevance for applications in neuroscience, deep-brain stimulation, and vehicle coordination \cite{dorfler2014synchronization}.

Note that definition \eqref{eq:splay}, which concerns the geodesic consensus system, does not use a parameter like $q$. If a closed curve $\gamma$ wraps around the area in its interior $q$ times, then the total length $l(\gamma)$ is $q$ times the length of a single loop.

\begin{proposition}\label{prop:equilibria}
Splay states as given by Definition \ref{eq:splay} are equilibria of the geodesic consensus system \eqref{eq:dxidt2} over the cycle graph $\mathcal{H}_N$ defined by \eqref{eq:HN}.
\end{proposition}

\begin{proof}
The dynamics of agent $i$ are
\begin{align}\label{eq:equilibrium}
\dot{x}_i=w_{i-1,i}\log_{x_i}(x_{i-1})+w_{i,i+1}\log_{x_i}(x_{i+1}).
\end{align}
Note that 
\begin{align*}
g_{x_i}(v,v)^\frac12|_{v=\log_{x_{i}}(x_{i+1})}&=d_g(x_{i},x_{i+1})=\frac{w_{i,i+1}^{-1}}{\sum_{j=1}^N w_{j,j+1}^{-1}}l(\gamma).
\end{align*}
Because $\gamma$ is a closed geodesic, the tangent vectors $\log_{x_{i}}(x_{i-1})$ and $\log_{x_i}(x_{i+1})$ are negatively aligned. There hence exists a $v\in\ts[\M]{x_i}$ such that
\begin{align}
\dot{x}_i
&=w_{i-1,i}\frac{w_{i-1,i}^{-1}}{\sum_{j=1}^N w_{j,j+1}^{-1}}l(\gamma)v-
w_{i,i+1}\frac{w_{i,i+1}^{-1}}{\sum_{j=1}^N w_{j,j+1}^{-1}}l(\gamma)v\nonumber\\
&=0.\tag*{\qedhere}
\end{align}
\end{proof}

Proposition \ref{prop:equilibria} shows that splay states are critical points of the disagreement function $V$. If the closed geodesic $\gamma$ is a local minimizer of the curve length function in the space $\Lambda(\M)/\!\sim$ given by \eqref{eq:quotient}, then we can say more: the set of splay states over $\gamma$ are local minimizers of $V$. Consider two types of local variations of $(x_i)_{i=1}^N$ around $\gamma$: (i) variations along $\gamma$ and (ii) variations that take the agents off $\gamma$. In case (i), minimality of $(x_i)_{i=1}^N$ is given by the equidistant partition of $\gamma$ when the weights $w_{i,i+1}$ are constant. This follows from the fact that for a given mean $\bar{t}:=\sum_{i=1}^nt_i/n$, the value of $\sum_{i=1}^nt_i^2$ is the least when all $t_i$ are equal, which is a special instance of a much more general symmetry principle in optimization \cite{waterhouse1983symmetric}. In case (ii), minimality of $(x_i)_{i=1}^N$ is given by the fact that $\gamma$ is a local minimizer of the curve length function. The proof of our main result, Theorem \ref{prop:main}, shows that we can combine these observations to cover all local variations of $(x_i)_{i=1}^N$.

\begin{theorem}\label{prop:main}
Let $(\M,g)$ be a geodesically complete Riemannian manifold. Suppose $\M$ contains a closed geodesic $\gamma$ of locally minimum length $L:=l(\gamma)$ in the space $\Lambda(\M)/{\sim}$ given by \eqref{eq:quotient}. Let $\gamma_i$ denote a geodesic from $x_i$ to $x_{i+1}$ and let
\begin{align*}
\mathcal{S}_\gamma:=\Bigl\{x\in\M^N\,\big|&{}\bigconc_{i=1}^N \gamma_i=\gamma,\\
&\,d_g(x_i,x_{i+1})<R-\varepsilon,\\
&\,d_g(x_i,x_{i+1})=\frac{w_{i,i+1}^{-1}}{\sum_{j=1}^N w_{j,j+1}^{-1}}L,
\,\forall\,i\in\V\Bigr\}
\end{align*}
denote a set of splay states on $\gamma$. Any element of $\mathcal{S}_\gamma$ is a local minimizer of $V$ given by \eqref{eq:V} with $\G=\mathcal{H}_N$. If $\gamma$ is a strict local minimizer of $l$, then $\mathcal{S}_\gamma$ is a strict local minimizer of $V$. 
\end{theorem}

\begin{proof} Because the closed geodesic $\gamma$ is a local minimizer of the curve length $l$ in the space $\Lambda(\M)/{\sim}$, there is an open ball $\mathcal{D}(\gamma)\subset\Lambda(\M)/{\sim}$ such that for all $\lambda\in\mathcal{D}(\gamma)$ it holds that $l(\lambda)\geq L$. Let $x=(x_i)_{i=1}^N\in\mathcal{S}_\gamma$. By continuity of the exponential map, due to $d_g(x_i,x_{i+1})<R-\varepsilon\leq r(x_i)$ for all $i\in\V$, all closed broken geodesics interpolating points in a neighborhood of $x$ are continuous. There is  a neighborhood $\mathcal{B}(x)\subset\M^N$ such that for all $y=(y_i)_{i=1}^N\in\mathcal{B}(x)$ it holds that the closed broken geodesic $\lambda(y)=\bigconc_{i=1}^N\lambda_i\in\mathcal{D}(\gamma)$, where $\lambda_i$ is a geodesic from $y_i$ to $y_{i+1}$, and $d_g(y_i,y_{i+1})<R-\varepsilon$. Hence $l(\lambda(y))\geq L$ and $V$ given by \eqref{eq:V} simplifies to the form \eqref{eq:V2} for all $y\in\mathcal{B}(x)$.
	
It suffices to prove local optimality of $x$ in $\mathcal{B}(x)$. The problem we wish to solve is
\begin{align}
\label{eq:intrinsic}
\min_{y\in\mathcal{B}(x)} V(y)&=\tfrac12\sum_{i=1}^N w_{i,i+1}d_g^2(y_i,y_{i+1}).
\end{align}
The constraint 
\begin{align}
	l(\lambda(y))=\sum_{i=1}^N d_g(y_i,y_{i+1})\geq L\label{eq:L}
\end{align}
holds on $\mathcal{B}(x)$. Since this constaint is redundant on $\mathcal{B}(x)$ we can add it to \eqref{eq:intrinsic} without any loss of optimality, forming
\begin{align}\tag{P}
	\label{eq:constrained}\begin{split}
		\min_{y\in\mathcal{B}(x)} V&=\tfrac12\sum_{i=1}^N w_{i,i+1}d_g^2(y_i,y_{i+1})\\
		\textrm{subject to }L&\leq\sum_{i=1}^N d_g(y_i,y_{i+1})\nonumber.
	\end{split}
\end{align}

%
%

Let a fixed $y\in\mathcal{B}(x)$ be a feasible solution to \eqref{eq:constrained}. Introduce a quadratic program,
\begin{align}
	\tag{QP${}_y$}
	\label{eq:QP}\begin{split}
		\min_{\ve{d}\in[0,\infty)^N} f&=\tfrac12\sum_{i=1}^N w_{i,i+1}d_{i,i+1}^2\\
		\textrm{subject to }l(\lambda(y))&=\sum_{i=1}^N d_{i,i+1}.
	\end{split}
\end{align}
Note that a feasible solution to \eqref{eq:QP} is given by the vector
\begin{align*}
\ve{d}(y)=[d_{i,i+1}(y)]:=[d_g(y_i,y_{i+1})]\in[0,\infty)^N
\end{align*}
and that this solution has the same objective value on \eqref{eq:QP} as the solution $y$ has on \eqref{eq:constrained}. 



We will show that $\ve{d}(y)$ is, in general, a suboptimal solution to \eqref{eq:QP}. Let $g:\mathcal{B}(x)\rightarrow\R$ denote the value of the optimal solution to \eqref{eq:QP} as a function of $y$. We will show that
\begin{align}
	V(y)=f(\ve{d}(y))\geq g(y)\geq g(z)|_{z\subset\gamma}=V|_{\mathcal{S}},\label{eq:train}
\end{align}
which implies optimality of $\mathcal{S}$. So far we have only shown the first relation. The second relation is true by the definition of $g$. It remains to establish the last two relations.

The positivity constraint $\ve{d}\in[0,\infty)^N$ in \eqref{eq:QP} can be relaxed. To see this, note that if $d_{j,j+1}<0$ for some $j\in\V$, then $d_{j,j+1}$ does not help to satisfy the constraint $l(\lambda(y))=\sum_{i=1}^N d_{i,i+1}$ while simultaneously incurring a cost $w_{j,j+1}d_{j,j+1}^2$ to the objective function value. An infeasible solution can be constructed where $d_{j,j+1}$ is replaced with $0$, which reduces the cost. To obtain a feasible solution, continue decreasing the values of other variables until the constraint holds. This results in the objective function value decreasing a second time, thus yielding a superior solution.

By relaxing the positivity constraints we obtain the equality constrained quadratic program
\begin{align}
	\tag{EQP${}_{y}$}
	\label{eq:QP2}\begin{split}
		\min_{\ve{d}\in\R^N} f&=\tfrac12\sum_{i=1}^N w_{i,i+1}d_{i,i+1}^2\\
		\textrm{subject to }l(\lambda(y))&=\sum_{i=1}^N d_{i,i+1}.\nonumber
	\end{split}
\end{align}
The optimal solution to \eqref{eq:QP2} is optimal to \eqref{eq:QP} and vice versa. There is no loss of generality in focusing on \eqref{eq:QP2}.

The optimization problem \eqref{eq:QP2} can be solved using the Karush-Kuhn-Tucker conditions for optimality \cite{nocedal1999numerical},
\begin{align}\label{eq:lagrange}
\begin{bmatrix}
	\ma{W} & \mat{1}\\
	\ma{1} & 0
\end{bmatrix}\begin{bmatrix}
\ve{d}\\
\mu
\end{bmatrix}=\begin{bmatrix}
\ve{0}\\
l(\lambda(y))
\end{bmatrix},
\end{align}
where $\ve{d}\in\R^{N}$ is given by $\ve[i]{d}=d_{i,i+1}$, $\ve{1}=[1\,\ldots\,1]\,\smash{\in\R^{N}}$, $\ma{W}$ with $\ma[ii]{W}=w_{i,i+1}$ is diagonal, and $\mu\in\R$ is a Lagrange multiplier. The solution is optimal by convexity of \eqref{eq:QP2}.


Denote 
\begin{align*}
\ma{A}:=\begin{bmatrix}
	\ma{W} & \vet{1}\\
	\ve{1} & 0
\end{bmatrix},\quad
\ma{M}:=\ma{W}\inv.
\end{align*}
It can easily be verified that
\begin{align*}
\ma{A}\inv=\frac{1}{\vphantom{1^1}\ve{1}\ma{M}\vet{1}}\begin{bmatrix}
	(\ve{1}\ma{M}\vet{1})\ma{M}-\ma{M}\vet{1}\ve{1}\ma{M} & \ma{M}\vet{1}\\\
	\ve{1}\ma{M} & -1
\end{bmatrix},
\end{align*}
and solving \eqref{eq:lagrange} yields  $\ve{d}=(l(\lambda(y))/\ve{1}\ma{M}\vet{1})\ma{M}\vet{1}$.
%
%
The value of the optimal solution to \eqref{eq:QP2} is
\begin{align}
\sum_{i=1}^N w_{i,i+1}d_{i,i+1}^2&=\vet{d}\ma{W}\ve{d}
=\frac{l^2(\lambda(y))}{\sum_{i=1}^N w_{i,i+1}^{-1}}.\label{eq:VLP}
\end{align}

Recall that $\gamma$ is a local minimizer of $l$ and $g(y)$ denotes the optimal value to \eqref{eq:QP2} and \eqref{eq:QP}. From \eqref{eq:VLP} it follows that 
\begin{align}\label{eq:gap}
g(y)=\frac{l^2(\lambda(y))}{\sum_{i=1}^N w_{i,i+1}^{-1}}\geq\frac{l^2(\gamma)}{\sum_{i=1}^N w_{i,i+1}^{-1}}=g(x)|_{x\in\gamma},
\end{align}
which is the third relation in \eqref{eq:train}. The last relation in \eqref{eq:train}, $g(x)|_{x\in\gamma}=V|_{\mathcal{S}}$, follows due to $\gamma$ being a closed geodesic whereby any set of points $\smash{x=(x_i)_{i=1}^N\subset\gamma}$ regenerates $\gamma=\lambda(x)$ as their closed broken geodesic. This property allows us to construct a solution $x(\ve{d})$ to \eqref{eq:constrained} from the solution $\ve{d}$ to \eqref{eq:QP2}. The optimal solution to the problem \eqref{eq:QP2} tells us how to position the agent $(x_i)_{i=1}^N$ on $\gamma$. The set of such points is $\mathcal{S}_\gamma$ by inspection of $\ve{d}$. Assume that $\gamma$ is a strict minimizer of $l$ in $\Lambda(\M)/{\sim}$, then the inequality in \eqref{eq:gap} is strict. 
\end{proof}



\begin{corollary}
Let $\M$ be closed and suppose Theorem \ref{prop:main} applies. If $\gamma$ is a strict local minimizer of $l$ in $\Lambda(\M)/{\sim}$, then $\mathcal{S}_\gamma$ is a Lyapunov stable equilibrium set of  \eqref{eq:dxidt}. If $\gamma$ is also an isolated critical point, then $\mathcal{S}_\gamma$ is asympotically stable.
\end{corollary}


\begin{proof}
Lyapunov stability follows from Proposition \ref{prop:not}.

Asymptotical stability also follows from Proposition \ref{prop:not} if $\mathcal{S}_\gamma$ is an isolated critical set of $V$. Assume that $\mathcal{S}_\gamma$ is not isolated. Let $x\in\mathcal{S}_\gamma$ and $\mathcal{B}_k(x)$ be a neighborhood of $x$ as detailed in the proof of Theorem \ref{prop:main}. Let $(\mathcal{B}_{k}(x))_{k=1}^\infty$, $\mathcal{B}_{k+1}(x)\subset\mathcal{B}_{k}(x)$ be a nonincreasing sequence of such neighborhoods of $x$ with $\lim_{k\rightarrow\infty}\mathcal{B}_{k}(x)=\{x\}$.
Since $\mathcal{S}_\gamma$ is not an isolated critical set, each $\mathcal{B}_k(x)\backslash\mathcal{S}_\gamma$ contains at least one critical point $y^k=(y^k_i)_{i=1}^N$ of $V$. Then $y^k$ satisfies
\begin{align}\label{eq:critical}
\dot{y}_i^k=w_{i-1,i}\log_{y_i^k}(y_{i-1}^k)+w_{i,i+1}\log_{y_i^k}(y_{i+1}^k)=0.
\end{align}
Let $\lambda_{i-1}^k$ and $\lambda_i^k$ denote the geodesics which connect $y_{i-1}^k,y_i^k$, and $y_{i+1}^k$. Equation \eqref{eq:critical} implies that the tangent vectors of  $\lambda_{i-1}^k$ and $\lambda_i^k$ are (negatively) aligned at $y_i^k$. As such, the curve $\lambda_{i-1}^k\conc\lambda_{i}^k$ is a geodesic. Moreover, $\lambda:=\bigconc_{i=1}^N\lambda_i^k$ is a closed geodesic by induction. Closed geodesics are critical points of $l$ in $\Lambda(\M)/{\sim}$. A geodesic $\lambda\notin[\gamma]$  being arbitrarily close to $\gamma$ contradicts the assumption that $\gamma$ is isolated.
\end{proof}


%



\subsection{Multistability}

\label{sec:extrinsic}


\noindent Multistability as given by Definition \ref{def:multistable} requires two sets of limits to be disjoint. The first is the consensus manifold:

\begin{proposition}\label{prop:consensus}
Suppose that $\M$ is closed. The consensus manifold $\mathcal{C}$ given by \eqref{eq:C} is a Lyapunov stable equilibrium manifold of Algorithm \ref{algo:intrinsic} and \ref{algo:extrinsic}.
\end{proposition}

\begin{proof}
Consider the disagreement function $W$ given by \eqref{eq:W} and assume $f$ is  either defined as in \eqref{eq:f} or by $f=1$. Then $W(x)\geq0$ and $W(x)=0$ if and only if $x\in\mathcal{C}$. The fact that $\M$ is closed implies that $\mathcal{C}$ is compact. The consensus manifold is Lyapunov stable by Proposition \ref{prop:not}.
\end{proof}



Note that geodesic consensus, Algorithm \ref{algo:intrinsic}, is trivially multistable since it is a bounded confindence model \cite{proskurnikov2018tutorial}. It is also multistable due to the presence of splay states, as described in Section \ref{sec:intrinsic}. Next, we show that this multistability extends from the cycle graph $\mathcal{H}_N$ to networks that are, roughly speaking, dominated by a large cycle.

Recall that $r(x)$ denotes the \emph{injectivity radius} at $x\in\M$, \ie the radius of the largest geodesic ball $\mathcal{B}_\varepsilon(x)\subset\M$ such that the exponential map $\exp_{x}$ is a diffeomorphism at $x$. Moreover, recall  $R:=\inf_{x\in\M}r(x)$. Suppose $\M$ is embedded in $\R^{n\times m}$ and let $\ma{X}:=\iota(x)$. Let $a(\ma{X})$ denote the radius of the largest ball defined in terms of the chordal distance,
\begin{align*}
\mathcal{A}_{a(\ma{X})}(\ma{X}):=\{\ma{Y}\in\M\,|\,\|\ma{X}-\ma{Y}\|\leq a(\ma{X})\},
\end{align*}
such that $\mathcal{A}_{a(\ma{X})}\subset\mathcal{B}_{r(x)}$. Let $A:=\inf_{\ma{X}\in\M}a(\ma{X})$. The following result relates $A$ to $R$.

\begin{lemma}\label{lem:balls}
Let $\M\subset\R^{n\times m}$ be a Riemannian manifold. For every pair $(\ma{X},\varepsilon)\in\M\times(0,\infty)$, any ball defined in terms of the geodesic distance,
\begin{align*}
\mathcal{B}_\varepsilon(\ma{X})=\{\ma{Y}\in\M\,|\,d_g(\ma{X},\ma{Y})\leq \varepsilon\}.
\end{align*}
contains a ball defined in terms of the chordal distance 
\begin{align*}
\mathcal{A}_{\delta(\varepsilon)}(\ma{X})=\{\ma{Y}\in\M\,|\,\|\ma{X}-\ma{Y}\|\leq\delta(\varepsilon)\},
\end{align*}
where the radius $\delta(\varepsilon)$ is strictly positive. Suppose $\M$ is closed, then $R=\inf_{\ma{X}\in\M}r(\ma{X})>0$ implies $A=\inf_{\ma{X}\in\M}a(\ma{X})>0$.

\end{lemma}

\begin{proof}
Suppose the first statement is false. Then, for some $\varepsilon>0$ and every $\delta_j>0$, there exists an $\ma[j]{Y}\in\mathcal{A}_{\delta_j}(\ma{X})$ with $d_g(\ma[j]{Y},\ma{X})>\varepsilon$. Form a sequence $(\ma[j]{Y})_{j=1}^\infty$ of such points where $\lim_{j\rightarrow\infty}\delta_j=0$. It follows that $\lim_{j\rightarrow\infty}\ma[j]{Y}=\ma{X}$, but $\lim_{j\rightarrow\infty}d_g(\ma[j]{Y},\ma{X})\geq  \varepsilon$. This contradicts the continuity of $d_g$.

Suppose $A=0$, then there is a sequence $(\ma[i]{Z})_{i=1}^N$ such that $\lim_{i\rightarrow\infty}a(\ma[i]{Z})=0$. If $\M$ is closed, then $(\ma[i]{Z})_{i=1}^N$ has a subsequence which converges to some  $\ma{Z}\in\M$. That $a(\ma{Z})=0$ contradicts the first result of this theorem.
\end{proof}




\begin{theorem}\label{th:multistable}
Let $\M$ be a closed, multiply connected, smooth Riemannian manifold such that $R=\inf_{x\in\M}r(x)>0$. Suppose there is a cycle subgraph, $\mathcal{H}_K\subset\G$, $K\leq N$, and an $(x_i)_{i=1}^N\in\M^N$ such that the closed broken geodesic generated by $(x_i)_{i\in\V(\mathcal{H}_K)}$ is not homotopic to a point on $\M$. Then the following two implications hold:
\begin{itemize}
\item[i)] For geodesic consensus, suppose $(x_i)_{i=1}^N$ at $t=0$ satisfies 
\begin{align*}
V(x)=\tfrac12\!\sum_{\{i,j\}\in\E}\!w_{ij}d_g(x_i,x_j)^2<\tfrac{1}{2}\min_{\{i,j\}\in\E}w_{ij}(R-\varepsilon)^2,
\end{align*}
where $\varepsilon<R$, then the flow \eqref{eq:dxidt} does not converge to $\mathcal{C}$.

\item[ii)] For chordal consensus, $\iota:\M\hookrightarrow\R^{n\times m}$, suppose $(\ma[i]{X})_{i=1}^N:=(\iota(x_i))_{i=1}^N$ at $t=0$ satisfies 
\begin{align*}
U(\ma{X})=\tfrac12\sum_{\{i,j\}\in\E}w_{ij}\|\ma[i]{X}-\ma[j]{X}\|^2<\tfrac{1}{2}\min_{\{i,j\}\in\E}w_{ij}A^2,
\end{align*}
then the flow \eqref{eq:dXidt} does not converge to $\mathcal{C}$.
\end{itemize}
\end{theorem}

\begin{proof}
Note that in the intrinsic case, 
\begin{align*}
d_g(x_i,x_j)^2\leq\tfrac{2}{w_{ij}}V<(R-\varepsilon)^2,
\end{align*}
wherefore $V$ is of the simplified form \eqref{eq:V2} at $t=0$. Moreover, $V$ remains of the simplified form for all future times since it decreases along flow lines. 

Likewise, for chordal consensus, the inequality $\|\ma[i]{X}-\ma[j]{X}\|<A$ holds for all
$t\in[0,\infty)$. By the definition of $A$, this implies that $d_g(\ma[i]{X},\ma[j]{X})<R$ for all $t\in[0,\infty)$. 

For both geodesic and chordal consensus, it follows that $\exp_{x_i}$ is a diffeomorphism on a set that includes $x_j$ for all $t\in[0,\infty)$. In particular, the geodesic $\gamma_{ij}$ connecting $x_i$ and $x_j$ is a continuous function of $t$. Moreover, the closed broken geodesic $\gamma$ interpolating $(x_i)_{i\in \V(\mathcal{H}_K)}$ is a continuous function of time for $t\in[0,\infty)$. The system evolution hence corresponds to a continuous deformation of $\gamma$. Suppose that the system reaches consensus, \ie that $\gamma$ converges to a single point. This contradicts the assumption of the theorem that $\gamma$ is not homotopic to a point at time $0$.\end{proof}





\begin{corollary}
The geodesic and chordal consensus systems, \eqref{eq:dxidt} and \eqref{eq:dXidt}, are multistable in the sense of Definition \ref{def:multistable} under the assumptions of Theorem \ref{th:multistable}.
\end{corollary}

\begin{proof}
By Proposition \ref{prop:consensus} there is some ball $\mathcal{B}_\delta(\mathcal{C})\subset\M^N$ from which the system cannot escape an arbitrarily small ball $\mathcal{B}_\varepsilon(\mathcal{C})\subset\M^N$. If we view the system as a set of $N$ points on $\M$, then, at any time $t$, there is a ball $\mathcal{B}_{\rho(\varepsilon)}(y(t))$ around some point $y(t)$ such that all  $x_i(t)\in\mathcal{B}_{\rho(\varepsilon)}(y(t))$.

If the system is not multistable, then, by Definition \ref{def:multistable} and Theorem \ref{th:multistable}, there is some time $\tau$ such that $\mathcal{B}_{\rho(\varepsilon)}(y(\tau))$ contains a closed curve that is not homotopic to a point. Note that ${\rho(\varepsilon)}\rightarrow0$ as $\varepsilon\rightarrow0$. For $\rho<R/2$, the open ball  $\mathcal{B}_{\rho(\varepsilon)}(y(\tau))$ is homotopy equivalent to $\R^d$, where $d$ is the dimension of $\M$. However, $\R^d$ cannot contain a curve that is not homotopic to a point, so we have a contradiction.\end{proof}

\begin{example}\label{ex:Lohe}
The cycle graph $\mathcal{H}_N$ satisfies the requirements of Theorem	\ref{th:multistable}, as do generalizations thereof such as the circulant graphs, $\mathcal{F}_{N,k}:=(\{1,\ldots,N\},\E_{N,k})$, with
\begin{align*}
\mathcal{E}_{N,k}:=\{\{i,j\}\,|\,|i-j|\leq k\},
\end{align*}
where we use addition modulo $N$. Note that $|\E_{N,k}|=kN$ being linear in $N$ is important. Position the agents so that the closed broken geodesic generated by $(\ma[i]{X})_{i=1}^N$ is an approximation of $\gamma$. More precisely, let the agents be approximately equidistantly spaced on small tubular neighborhood of $\gamma$ on $\M$ whereby $\|\ma[i]{X}-\ma[i\pm j]{X}\|\approx jl(\gamma)/N$. Then
\begin{align*}
U(\ma{X})=\tfrac12\!\!\!\!\!\!\sum_{\{i,j\}\in\E_{N,k}}\!\!\!\!\!\!w_{ij}\|\ma[i]{X}-\ma[j]{X}\|^2\approx\tfrac12\tfrac{l(\gamma)^2}{N^2}\!\!\!\!\!\!\sum_{\{i,j\}\in\E_{N,k}}\!\!\!\!\!\!w_{ij}\sum_{j=1}^k j^2.
\end{align*}
The assumptions of Theorem \ref{th:multistable} are satisfied for pairs $(N,k)$ with $N\gg k$ since $U\rightarrow0$ as $N\rightarrow\infty$ for any fixed $k$ under the assumption that $\sum_{\{i,j\}\in\E_{N,k}}w_{ij}\in \smallO(N^2)$ due to $|\E_{N,k}|=kN$. For example, with $w_{ij}=1$ we get $\sum_{\{i,j\}\in\E_{N,k}}w_{ij}=kN\in\mathcal{O}(N)\subset\smallO(N^2)$. Intuitively speaking, if we fix $k$, then we can keep increasing $N$ by $1$ and move the agents a little closer to each other each time until the assumptions hold.
\end{example}

\begin{example}\label{ex:final}
The Lohe models \eqref{eq:Lohe} and \eqref{eq:LoheSO2n} are multistable since $\mathsf{U}(n)$ and $\varphi\,\mathsf{U}(n)$ are multiply connected.
\end{example}

Note that Example \ref{ex:final} is a novel result that extends some findings in \cite{deville2018synchronization} from $\mathsf{SO}(n)$ to $\mathsf{U}(n)$. This extension is not trivial since $\mathsf{SO}(n)$ is a subset of zero measure in $\mathsf{U}(n)$.


%
%


\bibliographystyle{unsrt}
\bibliography{IEEEbib}


\end{document}